\documentclass[final,onefignum,onetabnum]{siuro210301}

\usepackage{lipsum}
\usepackage{amsfonts}
\usepackage{graphicx}
\usepackage{epstopdf}
\usepackage{algorithmic}
\ifpdf
  \DeclareGraphicsExtensions{.eps,.pdf,.png,.jpg}
\else
  \DeclareGraphicsExtensions{.eps}
\fi

\usepackage{enumitem}
\setlist[enumerate]{leftmargin=.5in}
\setlist[itemize]{leftmargin=.5in}

\newsiamremark{remark}{Remark}
\newsiamremark{hypothesis}{Hypothesis}
\crefname{hypothesis}{Hypothesis}{Hypotheses}
\newsiamthm{claim}{Claim}

\newcommand{\Hb}{\mathbf{H}}
\newcommand{\Cbb}{\mathbb C}
\newcommand{\Zbb}{\mathbb Z}
\newcommand{\Rbb}{\mathbb R}

\newcommand{\rank}{{\rm rank}}

\newcommand{\bmat}{\begin{bmatrix}}
\newcommand{\emat}{\end{bmatrix}}

\title{Spectral Factorization of Rank-Deficient Rational Densities 
\thanks{Submitted to the editors on Jan 13rd, 2023.}}

\author{Wenqi Cao\thanks{Department of Automation, Shanghai Jiao Tong University, Shanghai, China 
  (\email{wenqicao@sjtu.edu.cn}, \email{wnqcao@foxmail.com}).}
\and Anders Lindquist\thanks{Department of Automation and School of Mathematical Sciences, Shanghai Jiao Tong University, Shanghai, China. 
  (\email{alq@math.kth.se}).}
}

\usepackage{amsopn}

\begin{document}

\maketitle

\begin{abstract}
Though there are hundreds of papers on rational spectral factorization, most of them are concerned with full-rank spectral densities. In this paper we
propose a novel approach for spectral factorization of a rank-deficient spectral density, leading  to a minimum-phase full-rank spectral factor, in both the discrete-time and continuous-time cases. Compared with several  approaches to low-rank spectral factorization, our approach exploits a deterministic relation inside the factor, leading to  high computational efficiency. In addition, we show that this method is easily used in identification of low-rank processes and in Wiener filtering.
\end{abstract}

\begin{keywords}
spectral factorization, low-rank process, rank-deficient rational spectral densities, feedback representation 
\end{keywords}

\begin{MSCcodes}
{15A23, 46E20, 93E11, 93E12}
\end{MSCcodes}

\section{Introduction}
In this paper we consider spectral factorization of rational spectral densities of low rank, a topic for which there is a severe lack of computational methods compared with the situation for spectral factorization of polynomials or full-rank matrix densities.

Rational spectral densities often appear in second-order linear stochastic systems. Processes with a rank-deficient spectral density, named rank-deficiency processes or low-rank processes \cite{CLPtac21,CPLauto21}, may appear in dynamic networks 
where there are interconnections between the nodes \cite{WVD18-2,BGHP-17,CPcdc22} and play an important role in singular autoregressive (AR) models \cite{Deistler19,DeistlerEJC} as well as dynamic factor models \cite{FP17,Ferrante-20cdc}.
These system representations have recently attracted a lot of attention, especially in the large-scale cases,
in a broad range of areas, such as stochastic 
control \cite{ADCF12,LMP95}, macroeconomics \cite{LDAsurvey22}, engineering \cite{TR06,LSTB19}, biology and neuroscience \cite{nwbio09,Yuan-11}.
The need to calculate full-rank minimum-phase spectral factors thus increases rapidly, as they are used as the transfer function for a latent variable system, an innovation model, or similar.

The starting point of this research comes from our paper \cite{CPLauto21} on the identification of low-rank processes, where a full-column-rank minimum-phase factor is required for an innovation model. However for low-rank processes (see, e.g., \cite{CLPtac21}\cite{DeistlerEJC}), a classical identification approach like prediction-error methods (PEM, \cite{Ljung}) cannot be used directly because of the rank-deficient property.
In \cite{CPLauto21} we use a special feedback structure for low-rank systems to simplify  identification, and find that a minimum-phase full-rank spectral factor can be recovered under some restrictions, by means of a right-coprime factorization with an inner factor. This was not fully discussed in previous papers, but we found that our approach can be used to solve spectral factorization problem of rank-deficient densities analytically.

So far, there has been no less than hundreds of publications on the computation of spectral factors both in continuous-time and discrete-time cases; for references see, e.g., \cite{FG98,SK01}.
Classical methods based on matrix factorizations or filtering, like the Bauer method, the Levinson-Durbin algorithm, the Schur algorithm and the Riccati equation, can be found in \cite{Youla78,CG98,SK01,JLE11,ESS18,Ferrante05}.
For methods based on interpolation see, e.g., \cite{GK87,GK89}.
There are also several approaches based on the ideas underlying subspace identification \cite{VDDS97,LS19}.
Though the plethora of approaches can provide numerically efficient spectral factorization with sound properties, they typically require a common restrictive assumption of a positive definite spectral density, with the result that the methods cannot be applied to rank-deficient spectra.
It is worth mentioning that some so-called sparse spectral factorization methods \cite{11Sparse,13Sparse} consider  scalar polynomial spectral factorization with many zero data in signal processing and communications, which is different from rank-deficient spectral factorization.

To our best knowledge, only \cite{WM97,Oara00,Oara05,GCB06} propose computation methods for rank-deficient spectral factorizations to a full-rank minimum-phase factor.
Among them, \cite{WM97,GCB06} use (generalized) algebraic Riccati equations (ARE) to realize the factorizations iteratively. In the singular case, the inverse of a spectral density is not proper, resulting in the infeasibility of directly using the classical ARE methods.
Hence in \cite{WM97}, a full-rank descriptor form of the original spectra is used instead for ARE, which is of a higher dimension with some system pencils repeated on the diagonal of the realization denominator. Similarly, paper \cite{GCB06} overcomes the infeasibility by solving two different ARE at the same time, also increasing the calculation cost by multiples.
Worse still, as the dimension of the spectra increases considerably in large-scale problems, a huge amount of increased computation is required if choosing such schemes.

The paper \cite[Section V. B]{Oara00} introduces an approach to general spectral factorization (including low-rank spectral factorization) to obtain a full-rank minimum-phase factor in continuous time; for the same approach for the discrete-time case see \cite{Oara05}.
They solve the spectral factorization problem by first applying a left-coprime factorization with the same dimension as the spectral dimension, and then a general inner-outer factorization, without increasing the order of the equations to be solved like in the above approaches.

In this paper, we shall propose a novel coprime factorization-based rank-deficient spectral factorization approach to obtain full-rank minimum-phase factors.
Our main contributions are as follows.
Compared to the approaches in \cite{Oara00,Oara05}, our approach provides a coprime factorization with a lower order, and does not require a post-processing inner-outer factorization, leading to a higher computation efficiency.
Specifically, we extract the deterministic relation inside a low-rank spectral density, and use it to construct a coprime factorization problem, from which the set of all analytical minimum-phase spectral factors can be obtained after a matrix multiplication.
By using the deterministic relation, the factorization problem and hence the equation to be solved has a much lower dimension than for all existing methods above, and this makes it more efficient.
In addition, the identification of singular processes, and the Wiener filter between the sub-processes, are discussed through the spectral factorization results. These preliminary application results may provide insight into the further generalization of applications in stochastic control and filtering (see, e.g., \cite{FP19,Ferrante-20cdc,GCB06}) in the singular and large-scale cases.

The structure of this paper is as follows. The preliminaries for our approach are introduced in \cref{secPre}.  In particular, we investigate deterministic relations in tall spectral factors, consider the related problems of factorization, and establish the uniqueness of minimum-phase spectral factors.
The main observation behind our novel approach is described at the beginning of \cref{secSP}. Moreover, in \cref{secSP} our novel low-rank spectral factorization approach is explained in the discrete-time case, first focusing on a special case without operating the coprime factorization before describing the general solution. In \cref{secIden} we show the convenience of our approach in the application of identifying low-rank processes,  both for an innovation model, and for a canonical internal feedback structure with a Wiener predictor.
Section~\ref{secSPct} introduces low-rank spectral factorization in the continuous-time case. In \cref{secExamples}, numerical examples are given in discrete time with applications, as well as in continuous time.
Finally, the conclusions are given in \cref{secCon}.

\section{Preliminaries}\label{secPre}
We shall show in the next sections that, given the deterministic relation inside a low rank spectral density, the minimum-phase spectral factorization problem is solvable by a general coprime factorization procedure with an inner function. This was roughly discussed, but not dwelt on, in our previous work \cite{CPLauto21}.
In this section the preliminaries for solving rank-deficient spectral factorization in this paper are introduced.

\subsection{Deterministic relation in a tall spectral factor} 
Let $\Phi(z)$ be an $(m+p)\times (m+p)$ spectral density of rank $m$ in discrete time. 
By rearranging rows and columns, it can be partitioned as
\begin{equation}\label{Phi}
    \Phi(z)=\begin{bmatrix}
              \Phi_{11}(z) & \Phi_{12}(z) \\
              \Phi_{21}(z) & \Phi_{22}(z)
            \end{bmatrix},
\end{equation}
where $\Phi_{11}$ is $m\times m$ and full-rank, and $\Phi_{12}=\Phi_{21}^\top$.
It is well known that there exists an $(m+p) \times m$ full-column-rank (henceforth called merely full-rank) stable spectral factor of $\Phi(z)$ \cite[Remark 4.2.3]{LPbook},
\begin{equation}\label{W1W2}
    W(z)=\begin{bmatrix}
           W_1(z) \\
           W_2(z)
         \end{bmatrix},
\end{equation}
with $W_1(z)$ an $m \times m$ full-rank matrix, such that
\begin{equation}\label{Phifac}
    \Phi(z)=W(z)W(z)^*,
\end{equation}
where 
$W(z)^*=W(\bar{z}^{-1})^\top$ denotes the conjugate transpose, $m>0, p\geq 0$, and $m,p\in \Zbb$.

Our problem is how to extract a tall minimum-phase full-rank $W(z)$ from $\Phi(z)$ in \eqref{Phi}. Note that  spectral factorization of rational matrices are also studied in the dual case where $\Phi(z)=W^*(z)W(z)$. We choose the form \eqref{Phifac} because it is the natural factorization associated to the representation of second-order stationary stochastic processes and hence to filtering and estimation problems.

A deterministic relation between $W_1$ and $W_2$ was first proposed in \cite{GLsampling} and extracted from $\Phi(z)$ in \cite{CLPtac21,CPLauto21} with some conditions. Here we generalize the previous theorems to one without restrictions. The result also applies to the continuous-time case, and will be further used to calculate a minimum-phase full-rank $W$.

\begin{theorem}\label{lem1}
Suppose $W(z)$ is a full-column-rank spectral factor of $\Phi(z)$ in \eqref{Phi}, with partition \eqref{W1W2}, where $W_1$ is $m\times m$ and full-rank. Then there is a unique deterministic relation between $W_1$ and $W_2$, not affected by the particular choice of $W(z)$, namely
\begin{equation}\label{W1HW2}
  W_2(z)=H(z)W_1(z),
\end{equation}
where
\begin{equation}\label{H}
    H(z)=\Phi_{21}(z)\Phi_{11}(z)^{-1}.
\end{equation}
\end{theorem}

\begin{proof}
From \eqref{Phi} and \eqref{W1W2}, we have
$$
    \Phi_{11}=W_1W_1^*,~\Phi_{21}=W_2W_1^*.
$$
Since $W_1$ is full-rank, $\Phi_{11}$ is full-rank. Hence
$$
    W_2W_1^{-1}=\Phi_{21}\Phi_{11}^{-1},
$$
which leads to \eqref{W1HW2} and \eqref{H}.
Since the spectral density is unique for any process, the value of $H(z)$ is not affected by the particular choice of $W(z)$.
\end{proof}

\subsection{General left-coprime factorization with an inner factor}\label{subsec_copfact}
In the following, we shall give a corollary of \cite[Theorem 6.2]{Oara99} specialized to the left-coprime factorizations with an inner factor of the minimal degree in discrete time. And we will show in \cref{secSP} that the coprime factorization in our problem will only lead to a minimal degree solution.

We shall need the following notation for a realization
\begin{equation}\label{symb}
    \left[\begin{array}{c|c}A &B\\\hline C&D\end{array}\right]:= C(z I-A)^{-1}B +D.
\end{equation}

\begin{corollary}\label{thmCpFactdt}
Given an arbitrary rational matrix $T(z)$ with a minimal realization
\begin{equation}\label{minRlz}
    T(z)=\left[\begin{array}{cc|c}A_u & A_{us} & B_u \\ 0 & A_s & B_s\\
    \hline C_u & C_s&D\end{array}\right],
\end{equation}
where the eigenvalues of $A_u$, $A_s$ respectively correspond to the unstable (i.e., in the exterior of the closed unit disk
containing the infinity) and stable (i.e., in the closed unit disk) poles of $T(z)$.
Let $n_u$ be the number of unstable poles of $T(z)$.
Then the left-coprime factorization with an inner denominator with respect to the unit circle has a solution
\begin{equation}
T(z)=T_D(z)^{-1}T_N(z),
\end{equation}
of minimal degree $n_u$ if and only if the Stein equation
\begin{equation}\label{Steindt}
    X-A_u^*XA_u-C_u^*C_u=0
\end{equation}
has an invertible Hermitian solution X. In this case, the class of all solutions is given
by
\begin{subequations}
\begin{equation}
    T_N(z)=\left[\begin{array}{cc|c}
     A_u+R_u & A_{us}+R_s  & B_u+R_D\\
     0 & A_s & B_s\\ \hline
     PC_u & PC_s & PD \end{array}\right],
\end{equation}
\begin{equation}
    T_D(z)=\left[\begin{array}{c|c}
       A_u+R_u & (1-z)M \\
      \hline PC_u & P \end{array}\right],
\end{equation}
\end{subequations}
where
\begin{subequations}\label{RM}
\begin{equation}
    R_s=(1-z)MC_s,
\end{equation}
\begin{equation}
    R_u=(1-z)MC_u,
\end{equation}
\begin{equation}
    R_D=(1-z)MD,
\end{equation}
\begin{equation}
M=-X^{-1}(I-A_u)^{-*}C_u^*,
\end{equation}
and $P$ is an arbitrary unitary matrix.
\end{subequations}
\end{corollary}

The above corollary is easy to obtain by limiting the $J$ all-pass denominator to be an inner matrix in \cite[Theorem 6.2]{Oara99}, and specifying some matrices. Hence we omit the proof here. Note that our symbol \eqref{symb} of representing a realization is different from that of \cite{Oara99}.
Since we use right-coprime factorizations, matrix transpositions will be made after using the above result.
The continuous-time counterpart of this corollary is given in \cref{secSPct}.

\subsection{Remarks on calculating a minimal realization}\label{subsec_minrlz}
To solve the general coprime factorization problem as shown in \cref{subsec_copfact}, one may consider the way to obtain a minimal realization \eqref{minRlz}. Till now, the problem of calculating a minimal realization \eqref{symb} with 
matrices $A, B, C, D$ from a rational transfer function has been widely studied and led to numerous different approaches.
In this subsection, we shall give some suggestions on constructing \eqref{minRlz} given $T(z)$ with different kinds of poles. A numerical example in continuous time will be given in \cref{apdxExp}.

Considering here the matrix 
\begin{align}\label{AuAs}
    \begin{bmatrix}
        A_u & A_{us} \\ 0 & A_s
    \end{bmatrix}
\end{align}
which is block upper-triangular,
a convenient way to determine the realization is by using Gilbert realization (see, such as \cite[pp. 114-116]{Mackenrothbook}) when all the poles are distinct. Then \eqref{AuAs} will be a diagonal matrix with all poles on its diagonal. This also applies to the situation when there are complex distinct poles. 
Note that for our case, the unstable poles and the stable poles need to be put in $A_u$ and $A_s$ respectively.

When there is any pole with a multiple degree, Gilbert realization cannot be used directly anymore. However, the realization $T(z)$ can be seen as the parallel connection of different `smaller' realizations (see, such as \cite[pp. 116]{Mackenrothbook}). 
Specifically, we write $T(z)$ as a summation of products containing different poles (including their degrees), similar as when calculating Gilbert realization. 
Then the small realizations with distinct poles can be obtained as the above.
Separate the items with repeated poles, and for each repeated pole, construct its corresponding small realization with the matrix $A$ in \eqref{symb} as a Jordan block, with its diagonal the pole itself. Then the other matrices in the realizations are easily obtained. 

Though the above method is able to calculate a minimal realization in \eqref{minRlz}, there might be complex numbers in the matrices. A realization with complex matrices is not welcomed in the area of system control, for lack of physical meaning. Hence in this paper, we suggest constructing the small realizations with repeated complex poles, or poles symmetric on both sides of the real axis, in a different way from the above, such as using the standard controllable form. 

\subsection{Minimum-phase spectral factors and their uniqueness}\label{secUnique}
In the discrete-time formulation a rational spectral factor $W(z)$ is said to be {\em minimum-phase} if it has all its poles in the open unit disc and all its zeros in the closed unit disc \cite[p.137 or p.194]{LPbook}; in other words, $W(z)$ is outer. Next we shall demonstrate how established results on spectral factorization of full-rank spectral densities can be extended to the situation when the spectal factor is tall full-rank and minimum phase. More details on tall matrix zeros and minimum-phase function matrices are given in \cref{apdxA}.
 
As for uniqueness, there have been established results for full-rank factors in the dual case  $\Phi(z)=W^*W$.  In \cite{Youla61} it was proved in continuous time that a full row rank spectral factor is unique up to left multiplication by a unitary matrix, and \cite{BF16a} gave a corresponding result in discrete time.
Furthermore, \cite[Theorem 2.2]{BaggioThesis} proved such uniqueness of a full row rank minimum-phase factor in discrete time.
In the following, we shall modify these results to our case, which applies to the full-rank spectral densities as well. The proof will be given in \cref{apdxA.1}.

\begin{lemma}[{ Uniqueness of full-column-rank minimum-phase spectral factor}]\label{thmUnique}
A minimum-phase spectral factor $W(z)$ of a low rank spectral density $\Phi(z)$  always exists and is unique up to right multiplication by an arbitrary $m\times m$ constant unitary matrix.
\end{lemma} 

\section{Spectral factorization}\label{secSP}
In this section, we shall first explain the key observation behind our novel approach, and then provide a procedure for spectral factorization of $\Phi$ to obtain a full-rank minimum-phase factor.
Before starting, it is worth mentioning that neither $W_1$ nor $W_2$ in a minimum-phase full-rank factor $W$ need to be minimum-phase. This fact prevents us from obtaining $W$ by calculating some (maybe square or scalar) submatrices one by one; for more details see \cref{apdxA}.

\subsection{Key observation} 
Suppose we have already obtained a square minimum-phase factor $G_1$ for $\Phi_{11}$, which is easily obtained from any stable factor of $\Phi_{11}$ by square outer-inner factorization and existing computational methods. Hence for $W_1$ in a minimum-phase $W$, there exist a nontrivial $m \times m$ inner function $Q_1(z)$, satisfying
\begin{equation}\label{W1G1Q1}
    W_1(z)=G_1(z)Q_1(z).
\end{equation}
Then we have the following result on the existence of a rank-deficient spectral factorization solution.

\begin{lemma}\label{lemFeasible}
Given $H$, $G_1$, a full-column-rank minimum-phase spectral factor \eqref{W1W2} can be constructed.
\end{lemma}

\begin{proof}
From \eqref{W1G1Q1} and \eqref{W1HW2}, we have
\begin{equation}\label{CopInn}
    H(z)G_1(z)=W_2(z)Q_1^*(z)=W_2(z)Q_1(z)^{-1}.
\end{equation}
Since $W$ is minimum-phase, $W_2$ must be stable. Hence $Q_1^{-1}$ contains  the complete unstable part of the factor $HG_1$.
Hence one can get $W_2$ and $Q_1$  by performing a right-coprime factorization in the rational $\Hb^{\infty}_{p\times m}$ space, with a restriction that $Q_1$ should be inner. Hence, $W_1$ and $W$ can be calculated from \eqref{W1G1Q1}, \eqref{CopInn} and \eqref{W1W2} in both the continuous-time and discrete-time cases.
\end{proof}

\cref{lemFeasible} and \cref{thmUnique} respectively show the feasibility and uniqueness of such tall minimum-phase spectral factorizations through a right-coprime factorization.
That is, given a spectral density $\Phi$ with partition \eqref{Phi}, a minimum-phase full-rank factor can be determined uniquely up to right multiplication by a unitary matrix, from a minimum-phase factor $G_1$ of $\Phi_{11}$ and the deterministic relation $H$ inside the low rank process.

In the following, we shall give the set of solutions to tall full-rank minimum-phase factorization. In most cases, the solutions depend on solving a Sylvester matrix equation (specifically, a Stein equation for the discrete-time and a Lyapunov equation for the continuous-time case).
However, there is still a special case, when $W_1$ itself in a minimum-phase $W$ is minimum-phase. In that case the complicated right-coprime factorization can be avoided.
We shall first discuss the special case and then introduce the more general solution in discrete time.

\subsection{A special case}
When $W_1$ in a minimum-phase $W$ is itself minimum-phase, the matrix $H(z)G_1(z)=W_2(z)Q_1^*(z)$ will be a stable matrix, as shown in the proof of Theorem~\ref{thmHW1} below.
Hence we may obtain the set of minimum-phase $W(z)$, avoiding doing any right-coprime factorization. To explain this, we give the following theorem on the sufficient and necessary condition of a minimum-phase $W_1$.

\begin{theorem}\label{thmHW1}
Let $W(z)$ be a minimum-phase spectral factor with decomposition \eqref{W1W2} and let $H(z)$ be defined as in \eqref{W1HW2}. Then
$H(z)$ is stable if and only if $W_1(z)$ is minimum-phase.
\end{theorem}

\begin{proof}
Sufficiency: Recall that $H=W_2W_1^{-1}$. If $H$ is stable, any non-minimum-phase zeros of $W_1$ (i.e., the poles of $W_1^{-1}$), should be cancelled by the non-minimum-phase zeros of $W_2$.
However, if this holds, $W_1$ and $W_2$ will have the same non-minimum-phase zeros, implying that these zeros are non-minimum-phase zeros of $W$, which is conflict to the fact that $W$ is minimum-phase (see \cref{apdxA}). Hence $W_1$ has no non-minimum-phase zeros, i.e., $W_1$ is minimum-phase.\\
Necessity: When $W_1(z)$ is minimum-phase, by $W_1(z)=G_1(z)Q_1(z)$, similar to the reasons in the proof of \cref{thmUnique} in \cref{apdxA.1},
$Q_1$ should be a constant unitary matrix. Hence $W_2(z)Q_1^{-1}=H(z)G_1(z)$ is stable. Recalling that $G_1(z)$ has no non-minimum-phase zero, $H(z)$ is stable.
\end{proof}

From \cref{thmHW1}, given a stable $H(z)$, $W_1(z)$ will be minimum-phase when recovering a minimum-phase $W(z)$, i.e., $W_1(z)=G_1(z)Q_1$ with $Q_1$ a unitary matrix.
Hence the set of minimum-phase factors is
\begin{equation}\label{WwhenW1min}
    W(z)=\begin{bmatrix}
        G_1(z) \\
        H(z)G_1(z)
      \end{bmatrix}Q_1,
\end{equation}
where $Q_1$ is any unitary matrix.

\cref{thmHW1} applies directly to the continuous-time case, by substituting continuous-time functions.
It also gives a further answer to Manfred Deistler's question \cite{emailDeistler}: is there always a stable causal deterministic relation between the subvectors of a singular process (i.e., in the special feedback structure \eqref{fbspecial} in Section~\ref{secIden})? 
In our previous work \cite{CLPtac21}, this question was discussed and a negative answer was given through a counterexample.
In this paper a necessary and sufficient condition for $H(z)$ to be stable is given in Theorem~\ref{thmHW1}, which can be used as a judging criterion. Moreover, since $W_1(z)$ in a minimum-phase $W(z)$ can be non-minimum-phase no matter how $\Phi(z)$ is rearranged or partitioned, 
the stability cannot always hold.

In the following we shall give the procedure of obtaining a minimum-phase spectral factor $W(z)$ in general in discrete time, through the general left-coprime factorization methods in Corollary~\ref{thmCpFactdt}.
Note that by using the left-coprime results for our right-coprime factorizations, matrix transpositions are used in some final steps.

\subsection{Rank-deficient spectral factorization}
Suppose the proper function $G_1^\top H^\top$ has a minimal realization \eqref{symb},
where $A$ is square and invertible. From the fact that
$$
    \left[\begin{array}{c|c}A&B\\\hline C&D\end{array}\right] = \left[\begin{array}{c|c}PAP^{-1}& PB\\\hline CP^{-1}&D\end{array}\right],
$$
when $P$ is an invertible matrix,
the realization 
is equivalent to 
\begin{equation}\label{dtG1H}
    G_1(z)^\top H(z)^\top=\left[\begin{array}{cc|c}A_u & A_{us} & B_u \\ 0 & A_s & B_s\\
    \hline C_u & C_s&D\end{array}\right],
\end{equation}
where the eigenvalues of $A_u$, $A_s$ respectively correspond to the unstable 
and stable 
poles of $G_1(z)^\top H(z)^\top$. When \eqref{AuAs}
is a Jordan matrix, $A_{us}=0$. 

From Corollary~\ref{thmCpFactdt}, 
 we know that the solution to left-coprime (also our right-coprime) factorizations in discrete time 
 is based on solving the Stein equation \eqref{Steindt}.
Now we shall present our results on full-column-rank minimum-phase spectral factorization of rational rank-deficient densities.

\begin{theorem}\label{thmSPdt}
 Given a low-rank rational spectral density $\Phi(z)$ partitioned as in \eqref{Phi}, a minimum-phase square spectral factor $G_1(z)$ of $\Phi_{11}(z)$, and a minimal realization \eqref{dtG1H}, then the Stein equation \eqref{Steindt} has a unique invertible Hermitian solution $X$,
 and all the minimum-phase full column rank spectral factors $W(z)$ can be given from
 \begin{equation}
    W(z)=\begin{bmatrix}
           G_1(z)Q_1(z) \\
           W_2(z)
         \end{bmatrix},
 \end{equation}
with 
\begin{subequations}\label{W2Q1dt}
\begin{equation}\label{W2dt}
    W_2(z)=\left[\begin{array}{cc|c}
     A_u+R_u & A_{us}+R_s & B_u+R_D \\
     0 & A_s & B_s \\ \hline
     PC_u & PC_s & PD \end{array}\right]^\top,
\end{equation}
\begin{equation}\label{Q1dt}
    Q_1(z)=\left[\begin{array}{c|c}
        A_u+R_u & (1-z)M \\ \hline
        PC_u & P
    \end{array}\right]^\top,
\end{equation}
\end{subequations}
where $R_u$, $R_s$, $R_D$, $M$ are the same as \eqref{RM}, and $P$ is any unitary matrix; the other matrices are from the minimal realization in \eqref{dtG1H}. The pair $(Q_1,W_2)$ is unique when the unitary matrix $P$ has been chosen.
\end{theorem}

\begin{proof}
First we shall show that the solution $X$ to \eqref{Steindt} is unique and invertible, so that \cref{thmCpFactdt} with a minimal degree denominator is suitable for our problem to calculate the coprime factorization.
It is obvious that the pencils $A_u^*-zI$ and $A_u-zI$ are regular (i.e., the determinate of a pencil is not equivalent to $0$).
Then from \cref{thmSylvUnique} in \cref{apdxB}, since $A_u$ and $A_u^*$ are conjugate with $A_u$ containing only the poles outside the closed unit circle, $X$ is the unique solution to \eqref{Steindt}.
The invertibility of $X$ is established since \eqref{dtG1H} is a minimal realization, implying that $(C_u, A_u)$ is an observable pair.


Then from \cref{thmCpFactdt}, \eqref{W2Q1dt} is obtained.
Note that to obtain $W_2(z)$ and the inner denominator $Q_1(z)$, in the end a matrix transposition is required.

Next we shall show that \eqref{W2Q1dt} can represent all the solutions to minimum-phase full-rank spectral factorization, even though the minimal realization \eqref{dtG1H} is not unique. From \cref{thmUnique}, if we can give one spectral factorization solution, then all the solutions can be obtained by right multiplication by a unitary matrix.
The matrices in \eqref{W2Q1dt} can be written as
\begin{subequations}
 \begin{equation}\nonumber
 \begin{split}
    W_2(z)= &\left[\begin{array}{cc|c}
     A_u+R_u & A_{us}+R_s &  B_u+R_D\\
     0 & A_s & B_s\\ \hline
     C_u & C_s & D
     \end{array}\right]^\top \times P^\top \\
      = & \hat{W}_2(z) P^T,
\end{split}
 \end{equation}
 \begin{equation}\nonumber
 \begin{split}
     Q_1(z)&=\left[\begin{array}{c|c}
        A_u+MC_u(1-z) & M(1-z) \\ \hline
        C_u & I
    \end{array}\right]^\top\times  P^\top \\
    &=: \hat{Q}_1(z) P^\top,
 \end{split}
 \end{equation}
\end{subequations}
where $P$ is an arbitrary unitary matrix. Hence we have
$$
    W(z)=\begin{bmatrix}
           G_1(z)\hat{Q}_1(z) \\
           \hat{W}_2(z)
         \end{bmatrix} P^\top,
$$
which can represent all the minimum-phase full-rank factors.
\end{proof}

Note that the set of solutions will change when the partition of $\Phi$ as \eqref{Phi} changes.
Another proof, avoiding referring to \cref{thmUnique}, that \eqref{W2Q1dt} can represent all the solutions under some fixed partition, 
is given in \cref{apdxC}. 

When \eqref{AuAs} is in Jordan form, i.e., $A_{us}=0$, the result \eqref{W2dt} becomes
\begin{equation}\label{W2dtJordan}
\begin{split}
    W_2(z)=\left[\begin{array}{cc|c}
     A_u+R_u & R_s &B_u+R_D\\
     0 & A_s &B_s \\ \hline
     PC_u & PC_s &PD\end{array}\right]^\top,
\end{split}
\end{equation}
and other matrices remain the same as in Theorem~\ref{thmSPdt}.
When $H(z)$ is stable, there is no unstable part of the realization $G_1^\top H^\top$, hence \eqref{W2Q1dt} reduces to
$$
    W_2(z)=\left[\begin{array}{c|c}A_s&B_s\\\hline PC_s &PD\end{array}\right]^\top,\quad Q(z)=P^\top,
$$
with $P$ a unitary matrix, which leads to the same result \eqref{WwhenW1min} in the special case. Hence Theorem~\ref{thmSPdt} also works for the special case.
For the approaches to solve a Stein equation see, e.g., \cite{Chu87}.

A trick to simplify the calculations in practice is to choose a simple format or expression for the minimal realization of $G_1(z)^\top H(z)^\top$. This can be realized by restricting $A_u$ to be a Jordan block (even a diagonal matrix sometimes), or choosing $C$ with values suiting computation  so that \eqref{Steindt} and \eqref{W2Q1dt} are easy to solve and calculate. The trick works also for the continuous-time case in the following, and will be used in the examples of this paper in \cref{secExamples}.

Compared to the general spectral factorization methods in \cite{Oara00}\cite{Oara05}, we apply the coprime factorization to an $m\times p$ matrix, instead of a longer $m \times (m+p)$ matrix, by using the deterministic relation. Hence our approach is more efficient.

\section{Identification of low-rank vector processes}\label{secIden}
The above results in \cref{secSP} on solving the spectral factorization problem  also help us understand better the identification of low-rank vector processes.

Such low-rank processes may arise in diverse areas besides the research on control systems  \cite{FP19,Ferrante-20cdc,GCB06, ADCF12,LMP95}, namely macroeconomics \cite{LDAsurvey22}, networked systems \cite{BGHP-17, CPcdc22}, biology \cite{nwbio09,Yuan-11}, aviation \cite{TR06}, chemical industry \cite{milk22} and other fields. The low-rank vector processes are widely used, since they are common in practice when a system has interconnections, or when a large dimensional vector variable only depends on several scalar key elements. It is meaningful to study  estimation and identification specialized to such systems and thus break away from the traditional methods of full-rank systems which consume considerable computational resources and are not accurate enough in the low-rank cases.

The identification problems are discussed in several recent papers recently \cite{WVD18-2,BLM19,BGHP-17} including our previous papers \cite{PCL-21}\cite{CPLauto21}, where a preliminary but key problem can not be neglected: identifying an innovation model (for example, in discrete time)
\begin{equation}\label{yW1W2e}
  y(t)= W(z)e(t)=\begin{bmatrix}
          W_1(z) \\
          W_2(z)
        \end{bmatrix}e(t),
\end{equation}
where $y(t)$ of dimension $m+p$ is a low rank vector process (i.e., with a low rank spectral density), $e(t)$ of dimension $m$ is a normalized innovation process. Hence $W:=[W_1^\top,~W_2^\top]^\top$ is an $(m+p)\times m$ minimum-phase transfer matrix, also a full-rank minimum-phase spectral factor.

\subsection{An innovation model}
Since $y(t)$ has a greater dimension than $e(t)$, a classical identification approach such as prediction error methods (PEM) \cite{Ljung} cannot be applied directly. Then a problem comes naturally: can we identify the entries of $W$ one by one when it is $2\times 1$, or more generally, can we partition $W$ as in \eqref{yW1W2e}, with $W_1$ full-rank, and directly identify the two parts separately?

Recall that when identifying $y_1(t):=W_1(z)e(t)$, the model we estimate actually is $y_1(t)=G_1(z)e_1(t)$, where $G_1$ is minimum-phase and $e_1(t)$ is an innovation process of $y_1(t)$.
Hence our work \cite{CPLauto21} gave a negative answer to the above question from the aspect of a minimum-phase factor.
The reason in brief is that, neither $W_1$ nor $W_2$ must be minimum-phase, in a minimum-phase full-rank factor $W$ (for more details see \cref{apdxA}), implying that the estimates by PEM directly are not the ones of $W_1$ and $W_2$.
And how to recover the minimum-phase $W(z)$ from some accessible estimates becomes essential.

\begin{figure}[thpb]
      \centering
      \includegraphics[scale=0.45]{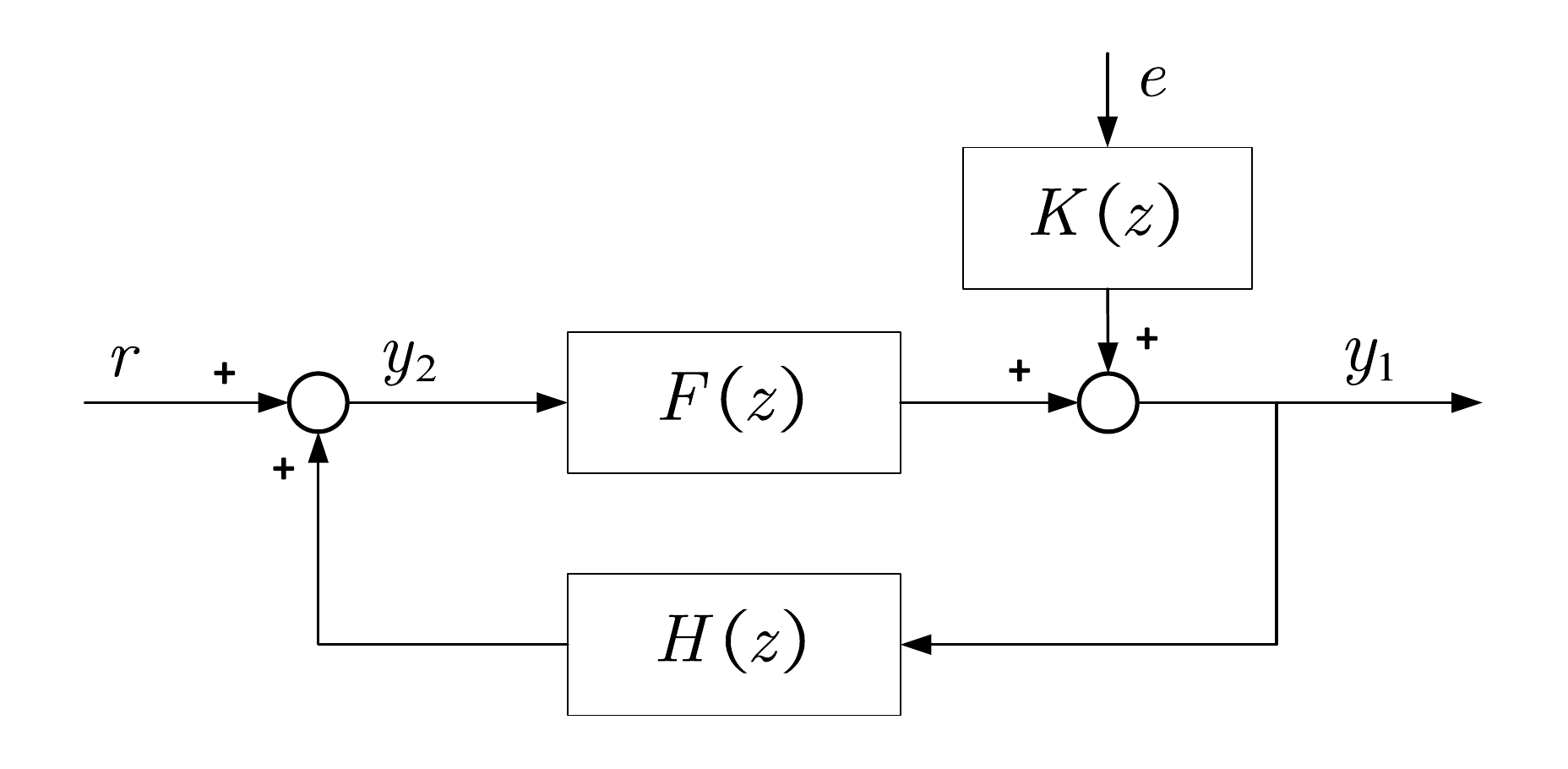}
      \caption {The structure diagram of feedback models.}
      \label{Figfb}
\end{figure}

In \cite{CPLauto21}, a special feedback structure as in \cref{Figfb} with $r(t)\equiv 0$ is used to simplify the identification of $W(z)$, and to explore the interconnections between the sub-vectors,
\begin{subequations}\label{fbspecial}
 \begin{align}\label{y1Fy2v}
        y_1(t) &= F(z)y_2(t)+ K(z)e(t), \\
    y_2(t) &= H(z)y_1(t),
\end{align}
\end{subequations}
where $K(z)$ is a minimum-phase function, $H(z)$ is the deterministic function in \cref{lem1}, which is
easy to identify by imposing a deterministic relation on $y_2$ and $y_1$. 

Previously, our approach had difficulties in estimating $W$ when facing processes with a larger scale, because there was a lack of sound theories on tall minimum-phase factors, and lack of simple and appropriate computational methods for low-rank spectral factorization.
In this paper, Theorem~\ref{thmHW1} first gives a necessary and sufficient condition to check if $W_1$ is minimum-phase. Then when the estimate of $H$ converges to a stable matrix, we can use the estimate of $G_1$ as the estimate of $W_1$, and an estimated innovation model \eqref{yW1W2e} can be calculated from \eqref{WwhenW1min}.
Conversely, if $H$ does not converge to a stable matrix, the general solution in \cref{secSP} can be used.

\subsection{The interconnections in the process}
As for the identification of the interconnections between $y_1(t)$ and $y_2(t)$, which are represented by the special feedback model \eqref{fbspecial}, we have shown that $H(z)$ is unique and hence identifiable given a fixed partition. However, the forward loop \eqref{y1Fy2v} is shown to be not identifiable in \cite{CPLauto21}.

To reconstruct $F(z)$, one idea is to construct a compensator $F$ for $H$ so that the whole system is internally stable, with the help of robust control (see, e.g., \cite{Zhou-D-G-95}). This was first suggested in our paper on modeling of low-rank time series \cite[Section VI, VIII-D]{CLPtac21}, which deals with an equivalent problem in dual form. There the problem was solved by Nevanlinna-Pick interpolation (see, e. g., \cite{CLtac}\cite{BLN03}). Note that here since $K(z)$ is a minimum-phase function, i.e., the process $K(z)e(t)$ is also an innovation process, the internal stability of the whole system only depends on whether the sensitive function is stable; see \cite{CLPtac21}.

Another view is to give a particular function $F(z)$ which for example, besides being stable with at least one unit delay. This $F(z)$ can act as the transfer function of the
Wiener predictor of $y_1(t)$ based on the (strict) past of $y_2(t)$, making
the two parts in the forward loop orthogonal.
In \cite{CPLauto21}, we found that when $K(z)$ in the forward loop is a constant matrix, $F(z)y_2(t)=zF_+(z)y_2(t-1)$ coincides with the one-step ahead Wiener predictor based on the strict past of $y_2(t)$, see \cref{lemcalF} in \cref{apdxA}. The estimation of a Wiener filter was also discussed, and realized in the scalar case.

Here in this paper, we shall first give a conclusive theorem on calculating the Wiener filter and its corresponding $K(z)$ in one forward loop, which we name a canonical forward loop. And then the calculations when $F(z)$ is a matrix are realized through our results above on the low rank spectral factorization to full column rank outer factors.

\begin{figure}[thpb]
      \centering
      \includegraphics[scale=0.6]{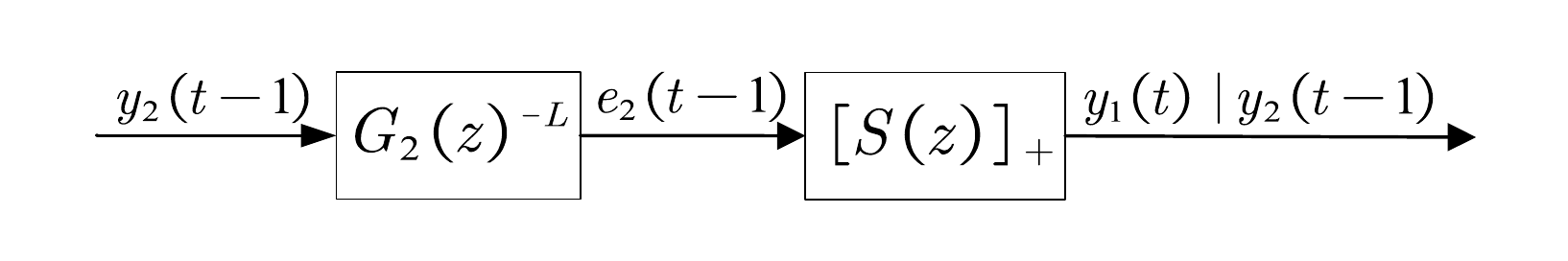}
      \caption {Strictly causal Wiener filter from $y_2$ to $y_1$.}
      \label{FigWiener}
\end{figure}

Suppose $W_2$ has an outer-inner factorization,
\begin{equation}\label{W2G2Q2}
    W_2(z)=G_2(z)Q_2(z),
\end{equation}
where $G_2$ is a $p\times r$ outer matrix, $Q_2$ is $r \times m$ inner satisfying $Q_2Q_2^*=I$, with $r=\rank(W_2)$.
A structure diagram of the Wiener filter is shown in Figure~\ref{FigWiener}, where $G_2(z)^{-L}$ denotes the left inverse of $G_2(z)$ and
$[\cdot]_+$ denotes the orthogonal projection operator onto the vector Hardy space $\Hb^2$ yielding the causal stable part of a function.
Then we have the following result.

\begin{theorem}[Wiener filter]\label{thmWiener}
A forward loop of the feedback model \eqref{fbspecial} can be given as
\begin{equation}\label{Feedplus}
    y_1(t)= F_+ y_2(t-1) +K_+e(t),
\end{equation}
where
\begin{subequations}
\begin{equation}\label{F+}
    F_+=[zW_1Q_2^*]_+G_2^{-L}
\end{equation}
denotes all solutions to the one-step ahead Wiener filter from $y_2$ to $y_1$,
\begin{equation}\label{K+}
    K_+=W_1-z^{-1}[zW_1Q_2^*]_+Q_2,
\end{equation}
\end{subequations}
$W$ is a minimum-phase spectral factor of the process $y$ as in \eqref{W1W2}, and $G_2$ is a minimum-phase full-rank spectral factor of $\Phi_{22}(z)$.
\end{theorem}

\begin{proof}
As in Figure~\ref{FigWiener}, a general one-step ahead Wiener filter is
$$
   \mathbb{E}\{y_1(t)\mid \Hb_{t-1}^-(y_2)\}= [S(z)]_+G_2(z)^{-L}y_2(t).
$$
where $S(z)$ denotes the map from $e_2(t-1)$ to $y_1(t)$, $G_2$ is a full-rank minimum-phase factor of $\Phi_{22}$ with $e_2(t)$ the corresponding innovation, i.e.,
\begin{equation}\label{y2G2e2}
y_2(t)=G_2(z)e_2(t).
\end{equation}
From \eqref{yW1W2e}\eqref{W2G2Q2} and \eqref{y2G2e2},
$$
    y_1(t)=W_1(z)e(t)=zW_1(z)Q_2^*(z)e_2(t-1),
$$
and hence we have $S(z)=zW_1Q_2^*$ and \eqref{F+}.
Then \eqref{K+} is obtained from
$$
    K_+e(t)=y_1(t)-F_+y_2(t)=(W_1-z^{-1}F_+W_2)e(t),
$$
from which the theorem follows.
\end{proof}

Note that the two parts of the right hand side of \eqref{Feedplus} are always orthogonal. From \eqref{y2G2e2}, $\Hb_t(y_2) = \Hb_t(e_2)$. Since $e_2(t)=Q_2(z)e(t)$ with $Q_2(z)$ inner, $\Hb_t(e_2) \subset \Hb_t(e)$.
Hence $F_+y_2(t-1)$ is in the projection of $\Hb_{t-1}(e_2)$ onto the space $\Hb^2$.
And $K_+e(t)$ denotes the other parts of $\Hb(y_1)$, including the ones in $\Hb_{t-1}(e_2) /\Hb^2$, $\Hb_t(e_2)$, and  $\Hb_t(e)/\Hb_t(e_2)$.

From \cref{thmWiener} we see that $W_2$ is needed in order to calculate the Wiener filter. This is easily realized when $W_1$ and $W_2$ are scalar, through a scalar coprime factorization.
Now with the help of \cref{thmSPdt}, $W_2$ can be obtained in matrix case, and hence a Wiener filter from $y_2$ to $y_1$ can be calculated in general.
Meanwhile, a forward loop with a strictly causal stable transfer function $F:=z^{-1}F_+$
is obtained as well, which can be applied to where (Granger) causality \cite{CainesBook} or the interconnections between nodes are needed \cite{BGHP-17}. 
 An instructive numerical example on identification and calculating the Wiener filter by using our spectral factorization approach will be given in \cref{exp3}.

\section{Spectral factorization in continuous time}\label{secSPct}
In the continuous-time setting $W(s)$ is {\em }minimum-phase, i.e. outer, if and only if all its poles are in the open left half plane and all its zeros are in the closed left half plane \cite[Chapter 5.3.1]{LPbook}.

In this section the symbol \eqref{symb} denotes
$$
    \left[\begin{array}{c|c}A &B\\\hline C&D\end{array}\right]:= C(sI-A)^{-1}B+D.
$$
Next we provide the continuous-time counterpart of \cref{thmCpFactdt}. 

\begin{corollary}\label{thmCpFactct}
Given an arbitrary rational matrix $T(s)$ with a minimal realization \eqref{minRlz}
where the eigenvalues of $A_u$, $A_s$ respectively correspond to the unstable (i.e., in the open right half plane) and stable (i.e., in the closed left half plane and at infinity) poles of $T(s)$. \\
Then the left-coprime factorization with an inner denominator with respect to the imaginary axis has a solution
\begin{equation}
\label{Scont}
T(s)=T_D(s)^{-1}T_N(s)
\end{equation}
of minimal degree $n_u$ if and only if the Lyapunov equation
 \begin{equation}\label{Lyapct}
     A_u^*X + XA_u -C_u^*C_u=0.
 \end{equation}
has an invertible Hermitian solution $X$. In this case, the class of all solutions is given by
\begin{subequations}
\begin{equation}
    T_N(s)=\left[\begin{array}{cc|c}
                 A_u+MC_u & A_{us}+MC_s & B_u+MD \\
                 0 & A_s & B_s \\ \hline
                 PC_u & PC_s & PD
               \end{array}\right],
\end{equation}
\begin{equation}
    T_D(s)=\left[\begin{array}{c|c}
           A_u+MC_u & M \\ \hline
           PC_u & P
         \end{array}\right],
\end{equation}
where
\begin{equation}
    M=-X^{-1}C_u^*,
\end{equation}
and $P$ is a unitary matrix.
\end{subequations}
\end{corollary}

Without loss of generality,
suppose the proper function $G_1(s)^\top H(s)^\top$ has a minimal realization
\begin{equation}\label{ctG1H}
    G_1(s)^\top H(s)^\top=\left[\begin{array}{cc|c}A_u & A_{us} & B_u \\ 0 & A_s & B_s\\
    \hline C_u & C_s&D\end{array}\right],
\end{equation}
where the eigenvalues of $A_u$, $A_s$ respectively correspond to the unstable and stable poles of $G_1(s)^\top H(s)^\top$.


Finally we present the continuous-time version of \cref{thmSPdt}.

\begin{theorem}\label{thmSPct}
Given a low-rank rational spectral density $\Phi(s)$ partitioned as in \eqref{Phi}, a minimum-phase square spectral factor $G_1(s)$ of $\Phi_{11}(s)$, and a minimal realization \eqref{ctG1H},
then the Lyapunov equation \eqref{Lyapct} has a unique invertible Hermitian solution $X$,
and all the minimum-phase full column rank spectral factors $W(s)$ are given by
\begin{equation}
    W(s)=\begin{bmatrix}
           G_1(s)Q_1(s) \\
           W_2(s)
         \end{bmatrix},
\end{equation}
with
 \begin{subequations}\label{W2Q1ct}
 \begin{equation}\label{W2ct}
    W_2(s)=\left[\begin{array}{cc|c}
     A_u+MC_u & A_{us}+MC_s & B_u+MD \\
     0 & A_s & B_s\\ \hline
     PC_u & PC_s & PD
     \end{array}\right]^\top
 \end{equation}
 \begin{equation}\label{Q1ct}
    Q_1(s)= \left[\begin{array}{c|c}
        A_u+MC_u & M \\ \hline
        PC_u & P
    \end{array}\right]^\top
 \end{equation}
 where
 \begin{equation}
    M=-X^{-1}C_u^*,
 \end{equation}
$P$ is an arbitrary unitary matrix, and the other matrices are as in the minimal realization \eqref{ctG1H},
where
\begin{equation}
    H(s)=\Phi_{21}\Phi_{11}^{-1}.
\end{equation}
 \end{subequations}
\end{theorem}

\begin{proof}
That the Lyapunov equation \eqref{Lyapct} has a unique solution can be shown by \cref{thmSylvUnique}.
Then \eqref{W2Q1ct} can be obtained from \cref{thmCpFactct}.
The other parts of this proof are similar to those in the proof of \cref{thmSPdt}, and therefore we omit them here.
\end{proof}

Approaches to solve a Lyapunov equations can be found in \cite{Chu87}\cite{BS72}.
Moreover, Theorem~\ref{thmSPct} works also for the special case when $H(s)$ is stable, by deleting the unstable parts in \eqref{W2Q1ct}.

\section{Numerical Examples}\label{secExamples}
In this section we shall give three examples illustrating the theoretical results above. First we shall give an example of the special case of subsection 3.2, showing the convenience of this necessary and sufficient condition.
Next, more general spectral factorization examples separately in continuous-time case and discrete-time case will be given.
In the example for discrete-time case, the spectral factorization will be applied to a singular process identification problem, where the forward loop of the feedback structure with $F_+(z):=zF(z)$ a one-step Wiener filter is calculated as well.

\subsection{Example 1: a special case in the discrete time}

Suppose we are given a minimum-phase factor $G_1(z)$ of $\Phi_{11}(z)$ and the deterministic relation function $H(z)$ as
$$
    G_1(z)=\begin{bmatrix}
             \frac{(z+0.4)(z+0.3)}{(z+0.2)(z+0.1)} & 0 \\
             0 & \frac{z+0.3}{z+0.5}
           \end{bmatrix}, \quad H(z)=\begin{bmatrix}
                                       \frac{z+0.1}{z+0.3} & 1
                                     \end{bmatrix}.
$$
$H(z)$ has only one stable pole $-0.3$, hence it is stable.
From Theorem~\ref{thmHW1}, a minimum-phase factor $W(z)$ can be directly given by
$$
    W(z)=\begin{bmatrix}
           G_1(z) \\
           H(z)G_1(z)
         \end{bmatrix}=\begin{bmatrix}
             \frac{(z+0.4)(z+0.3)}{(z+0.2)(z+0.1)} & 0 \\
             0 &  \frac{z+0.3}{z+0.5} \\
             \frac{z+0.4}{z+0.2} & \frac{z+0.3}{z+0.5}
           \end{bmatrix},
$$
without solving a coprime factorization.
And any full column rank minimum-phase factor of this density can be obtained from $W(z)Q_1$, where $Q_1$ is a unitary constant matrix.

The poles of the above $W(z)$ are $-0.5, -0.2, -0.1$, and the zeros are $-0.4, -0.3$, showing that $W(z)$ is minimum-phase, in harmony with Theorem~\ref{thmHW1}.

\subsection{Example 2: continuous-time case}
In this example, we shall introduce some tricks of calculation in the continuous-time case.  To illustrate the generality of our approach, we have  an example with different types of poles.

Suppose we start from a minimal stable factor $W_o(s)=[{W}_{o1}(s)^\top {W}_{o2}(s)^\top]^\top$,
\begin{equation}\nonumber
W_{o1}(s)=\begin{bmatrix}
  \frac{(s+1)(s-2)(s^2-2s+2)}{(s+3)(s+4)(s^2+2s+2)} &  0 \\
  0 & \frac{(s+3)(s-1)}{(s+1)(s+5)}
\end{bmatrix},
\end{equation}
\begin{equation}\nonumber
 W_{o2}(s)=\begin{bmatrix}
              \frac{(s+1)^2(s-2)(s+2)}{(s+4)^2(s^2+2s+2)}   &
              \frac{s+3}{s+4}
            \end{bmatrix},
\end{equation}
with ${W}_{o1}$ full-rank.

First we impose an outer-inner factorization on ${W}_{o1}(s)$ without here dwelling on the computational method to obtain ${W}_{o1}(s)=G_1(s)Q(s)$, where
$$
    G_1(s)=\begin{bmatrix}
             \frac{(s+1)(s+2)}{(s+3)(s+4)} & 0 \\
             0 & \frac{s+3}{s+5}
           \end{bmatrix}
$$
is minimum-phase and 
$$
 Q(s)=\begin{bmatrix}
           \frac{(s^2-2s+2)(s-2)}{(s^2+2s+2)(s+2)} & 0 \\
           0 & \frac{s-1}{s+1}
           \end{bmatrix}
$$
is an inner function matrix.

When the expression of ${W}_{o1}$  is more complex than that of $G_1$, of which we need to calculate the inverse, we may transform the matrix ${W}_{o}(s)$ into ${W}_{o}(s)Q^*(s)$ to calculate $H(s)$. This is feasible because the deterministic function $H(s)$ remains the same as the spectral factor changes (Theorem~\ref{lem1}).
By transformation,
\begin{equation}\label{ex2HG1}
    {W}_{o2}(s)Q^*(s)=\begin{bmatrix}
                   \frac{(s+1)^2(s+2)^2}{(s+4)^2(s^2-2s+2)} & \frac{(s+1)(s+3)}{(s-1)(s+4)}
                 \end{bmatrix},
\end{equation}
and then $H(s)$ is easily obtained by
\begin{equation}\nonumber
\begin{split}
    H(s)&=G_1(s)^{-1}{W}_{o2}(s)Q^*(s)\\
    &= \begin{bmatrix}
         \frac{(s+1)(s+2)(s+3)}{(s+4)(s^2-2s+2)} & \frac{(s+1)(s+5)}{(s-1)(s+4)}
       \end{bmatrix}.
\end{split}
\end{equation}

However in this case, given ${W}_{o}$ and the outer-inner factorization results, we may skip the step of calculating $H(s)$.
That is, we have $H(s)G_1(s)={W}_{o2}(s)Q^*(s)$ as in \eqref{ex2HG1} directly, with distinct unstable poles $1+i, 1-i, 1$, and a stable pole $-4$ with degree $2$. From the calculations in \cref{apdxExp}, a possible minimal realization of $G_1(s)^\top H(s)^\top$ is \eqref{exp_ctrlz},
where clearly 
$$
    A_u=\begin{bmatrix}
        0 & 1 & 0 \\ -2& 2& 0 \\ 0 & 0 & 1
    \end{bmatrix},
    ~A_s=\begin{bmatrix}
                 -4 & 1 \\
                 0 & -4
               \end{bmatrix},~A_{us}={\bf{0}}_{3*2},~B_u=\begin{bmatrix}
                   0 \\ 25/169 \\ 8/5
               \end{bmatrix},~B_s=\begin{bmatrix}
         10/13\\
           -3/5
          \end{bmatrix},
$$
$$
    C_u=\begin{bmatrix}
          -7 & 12& 0\\
           0 & 0 & 1
        \end{bmatrix},~C_s=\begin{bmatrix}
                             -30/13& 0 \\
                             0 & 1
                           \end{bmatrix},~D=\begin{bmatrix}
         1 \\
          1
          \end{bmatrix}.
$$
Plugging $A_u$ and $C_u$ into \eqref{Lyapct}, we have the unique Hermitian solution to the Lyapunov equation
$$
X=\begin{bmatrix}
     99/4& -49/4 & 0 \\
     -49/4& 337/8 &0 \\
    0 & 0 & {1}/{2}
\end{bmatrix}.
$$
Hence from Theorem~\ref{thmSPct} and equation \eqref{W2Q1ct}, the set of the solutions  is
$$
    M=\begin{bmatrix}
                      28/169 & 0\\
                     - 40/169 & 0\\
                     0 & -2
                   \end{bmatrix},
$$
\begin{equation}\nonumber
    W_2(s) =\begin{bmatrix}
        \displaystyle{\frac{(s+2)^2(s+1)^2}{(s+4)^2(s^2+2s+2)}}, & 
        \displaystyle{\frac{s+3}{s+4}}
      \end{bmatrix} P^\top,
\end{equation}
\begin{equation}\nonumber
    Q_1(s)=\begin{bmatrix}
        \frac{s^2-2s+2}{s^2+2s+2} & 0  \\
        0 & \frac{s-1}{s+1}
    \end{bmatrix}P^\top,
\end{equation}
where $P$ is any $2\times 2$ unitary matrix. Choosing $P$ to be the identity matrix, we have
the minimum-phase spectral factor,
\begin{equation}\nonumber
\begin{split}
    W(s)=\begin{bmatrix}
           G_1(s)Q_1(s) \\
           W_2(s)
         \end{bmatrix}=\begin{bmatrix}
                        \frac{(s+1)(s+2)}{(s+3)(s+4)} & 0 \\
                         0 & \frac{(s-1)(s+3)}{(s+1)(s+5)} \\
                         \frac{(s+2)^2(s+1)^2}{(s+4)^2(s^2+2s+2)} & \frac{s+3}{s+4}
                       \end{bmatrix},
\end{split}
\end{equation}
with poles $-1, -3, -4, -4, -5$ and zeros $-1, -2, -3$.

Note that if $H$ is given, $W_2$ can also be calculated from $HG_1Q_1$ instead of \eqref{W2ct}.
This continuous-time example shows that when we start from a non-minimum-phase factor, the computation can be further simplified by cutting down calculating the inverse of an $m\times m$ matrix, either by skipping calculating $H$ or by calculating $W_2=HG_1Q_1$ instead of \eqref{W2ct} (\eqref{W2dt} in discrete time).

\subsection{Example 3: discrete-time case and the application to identification}\label{exp3}
In this example, we shall solve the minimum-phase spectral factorization problem in discrete-time through Theorem~\ref{thmSPdt}, in an application of a singular process identification.
For simplicity, we identify a singular process of rank $1$.
Note that in the following,  $G_1(z)$, $H(z)$, $W(z)$, etc. denote estimates rather than functions, but are kept for theoretical clarity.

Consider a three-dimensional process $y(t)$ of rank $1$ described by
$$
    y(t)=W_o(z)e(t)=\begin{bmatrix}
                      W_{o1}(z) \\
                      W_{o2}(z)
                    \end{bmatrix}e(t),
$$
where $e$ is a zero mean white Gaussian scalar noise of variance $\lambda^2 = 1$, and the two blocks of function $W_o(z)$ are
$$
W_{o1}(z)= \frac{z+2}{5z-1}, \quad W_{o2}(z)=\begin{bmatrix}
  \frac{z-2}{5z-1} &
  \frac{z-1}{5z-1}
\end{bmatrix}^\top.
$$
From these we obtain the transfer function
$$
    H_o(z)=\begin{bmatrix}
           \frac{z-2}{z+2} & \frac{z-1}{z+2}
         \end{bmatrix}^\top.
$$
We have generated 100 samples of the three-dimensional time series with $N = 500$ data points $\{y(t):=[y_1(t), y_2(t)^\top]^\top\in \Rbb^3;~t=1, \cdots, N\}$ in MATLAB, where $y_1(t)$ denotes the first scalar entry of $y(t)$, and $y_2(t)$ the remaining two in sequence.
With these data we shall successively identify a minimum-phase factor for the process $y_1$, the deterministic relation function $H_o(z)$, and calculate the estimates of a minimum-phase factor of $y$ as well as the Wiener fliter from $y_2$ to $y_1$.

Suppose the orders of functions are known. First we estimate an ARMA model of $y_1$ and obtain
$$
    \hat{y}_1(t)-0.255y_1(t-1)=e_1(t)+0.528e_1(t-1),
$$
with $e_1(t)$ an innovation of $y_1(t)$. Hence we have the estimate
$$
    G_1(z)=\frac{z+0.528}{z-0.255},
$$
which is minimum-phase.
Then identifying $H_o(z)$ by imposing a deterministic relation between the data of $y_1(t)$ and $y_2(t)$,
a consistent and nearly precise estimate is given by a least squares method,
$$
    H(z)=\begin{bmatrix}
           \frac{z-2.000}{z+2.000} & \frac{z-1.000}{z+2.000}
         \end{bmatrix}^\top.
$$
For more details on identification, see our work \cite{CPLauto21}.

Then we use Theorem~\ref{thmSPdt} to calculate the estimate of a minimum-phase factor of $y(t)$.
A minimal realization of $G_1(z)^\top H(z)^\top$ can be given as
\begin{equation}\nonumber
\begin{split}
    G_1(z)^\top H(z)^\top = \left[\begin{array}{cc|cc}
        -2 & 0  & 2.611 & 1.958  \\
        0 & 0.255  & 0.606 & 0.259 \\ \hline
        -1 & -1 & 1 & 1
    \end{array}\right],
\end{split}
\end{equation}
where we choose $C=[-1,-1]$ to simplify the calculation. The matrices $A_u$, $A_s$, $B_u$, etc. are obtained by a procedure, omitted here, similar to the one in Example 2. The unique solution to the Stein equation \eqref{Steindt} is
$$
    X=-1/3.
$$
Hence from Theorem~\ref{thmSPdt}, by setting $P=I$, we have
$$
    M=-1, \quad Q_1(z)=\frac{z+2}{2z+1}
$$
and one of the estimated minimum-phase factors is
\begin{equation}\nonumber
\begin{split}
    W(z) &=\begin{bmatrix}
            1 \\
            H(z)
          \end{bmatrix}G_1(z)Q_1(z)\\
          &=\frac{(z+0.528)}{(2z+1)(z-0.255)}\begin{bmatrix}
              z+2 \\ z-2 \\ z-1
            \end{bmatrix},
\end{split}
\end{equation}
with poles $-0.5, 0.255$ and a zero $-0.528$.

Next we will use the above results to estimate a canonical forward loop of the feedback structure with a Wiener filter.
From the above we have the partition
$$
    W_1(z)= \frac{(z+2)(z+0.528)}{(2z+1)(z-0.255)},
$$
$$
    W_2(z)=\begin{bmatrix}
             \frac{(z-2)(z+0.528)}{(2z+1)(z-0.255)} & \frac{(z-1)(z+0.528)}{(2z+1)(z-0.255)}
           \end{bmatrix}^\top,
$$
where $W_2$ is minimum-phase. Hence an outer-inner factorization of $W_2$ is $W_2=G_2Q_2$, with $G_2=W_2$, $Q_2=1$.
Then from Theorem~\ref{thmWiener}, a one-step ahead Wiener filter from $y_2$ to $y_1$ is
\begin{equation}\nonumber
\begin{split}
    F_+(z)&=[zW_1(z)Q_2(z)^*]_+W_2(z)^{-L}\\
        &= \frac{z(2.283z+1.184)}{(z+0.528)(3z^2-7z+3)} \begin{bmatrix}
                                                       2z-1, & z-1
                                                     \end{bmatrix}.
\end{split}
\end{equation}
by defining the pseudo-inverse $W_2^{-L}=(W_2^*W_2)^{-1}W_2^*$,
where
\begin{equation}\nonumber
    [zW_1(z)]_+=\frac{z(2.283z+1.184)}{(2z+1)(z-0.255)}.
\end{equation}
Then from \eqref{K+},
\begin{equation}\nonumber
\begin{split}
   K_+=W_1-z^{-1}[zW_1]_+=0.5.
\end{split}
\end{equation}
Since $F:=z^{-1}F_+$ is strictly causal, it is easy to verify that $K_+e(t)$ is orthogonal to $F(z)y_2(t)$.

\section{Conclusion}\label{secCon}

A novel low rank rational spectral factorization approach is proposed in this paper. The deterministic relation in the spectral factor and the coprime factorization with an inner factor are used to calculate the full-rank minimum-phase spectral factor efficiently. The application of the algorithm in identifying low-rank processes is introduced, where an innovation model and also the internal Wiener filter can be estimated. Examples show the feasibility and convenience of our approach.

\appendix
\section{Details on matrix zeros and tall minimum-phase matrix functions}\label{apdxA}
In this section we give more details on zeros and other properties of tall minimum-phase matrix functions.

\begin{definition}[zeros of a function matrix] Given an $m\times p$ function matrix $W(\lambda)$, a complex number $\alpha$ in the region where $W$ is analytic is a (right) zero of $W$ if there is a nonzero vector $v\in \Cbb^p$, such that
$$
    W(\alpha)v=0.
$$
\end{definition}

Consider now the outer-inner factorizations
\begin{subequations}
\begin{align}\nonumber
  W_1(\lambda) &= G_1(\lambda)Q_1(\lambda),  \\ \nonumber
  W_2(\lambda) &= \hat{G}_2(\lambda)\hat{Q}_2(\lambda),
\end{align}
\end{subequations}
where $G_1$, $\hat{G}_2$ are the outer (minimum-phase) factors and  $Q_1, \hat{Q}_2$ are square inner (in fact matrix Blaschke products). The question we want to discuss is: if $W$ is outer, does it follow that any (or both) of the two components $W_1, W_2$ are also outer? We shall see that the answer is in general negative.

Let us recall that the full-column-rank matrix function $W(z)\in \Hb_{(p+m),m}^2$ is outer, if the row-span
$$
    \overline{\rm span}(W) :=\overline{\rm span}\{\phi(z)W(z);~\phi\in\Hb^{\infty}_{(p+m)}\}
$$
is the whole space $\Hb^2_m$.
 The {\it  greatest common right  inner divisor} of two inner functions $Q_1$  and $ \hat{Q}_2 $, see \cite[p. 188 top]{Fuhrmann-86}  is denoted $Q_1\wedge_R \hat{Q}_2$. This is the inner function representative of the closed vector sum $\Hb^2_{m}Q_1\vee \Hb^2_{m}\hat{Q}_2$.

\begin{theorem}
Let a full-column-rank matrix function  $W(z)\in \Hb_{(p+m),m}^2$ be partitioned as in \eqref{W1W2}. Then $W$ is outer if and only $Q_1$  and $\hat{Q}_2$ are right-coprime, i.e.,  the {\it greatest common right  inner divisor} of $Q_1$  and $ \hat{Q}_2 $  is the identity, i.e. $Q_1 \wedge_R \hat{Q}_2=I_m$.
\end{theorem}
\begin{proof}
Since the square full-rank outer matrix function $G_1$ satisfies $\overline{\rm span}(G_1)=\Hb_m^2$, we have
\begin{equation}\nonumber
\begin{split}
    \overline{\rm span}(W) &=\overline{\rm span}(G_1)Q_1 \vee \overline{\rm span}(\hat{G}_2)\hat{Q}_2 \\
    &= \Hb_m^2Q_1 \vee \overline{\rm span}(\hat{G}_2)\hat{Q}_2
\end{split}
\end{equation}
where $\hat{G}_2$ is a $p\times m$ outer matrix, possibly not full-rank, satisfying $\overline{\rm span}(\hat{G}_2) \subset \Hb_m^2$. Hence,
$$
    \Hb_m^2Q_1 \vee \overline{\rm span}(\hat{G}_2)\hat{Q}_2~\subset~\Hb_m^2Q_1 \vee \Hb_m^2\hat{Q}_2.
$$
On the other hand,
$$
    \Hb^2_{m} (Q_1 \wedge_R \hat{Q}_2)~\subset~ \Hb^2_{m} Q_1 ~\subset~ \Hb_m^2Q_1 \vee \overline{\rm span}(\hat{G}_2)\hat{Q}_2.
$$

Follows from the identity see \cite[p. 188 top]{Fuhrmann-86},
$$
\Hb^2_{m}Q_1\vee \Hb^2_{m}\hat{Q}_2=  \Hb^2_{m} (Q_1 \wedge_R \hat{Q}_2),
$$
and from the above, we have
$$
    \overline{\rm span}(W)=\Hb^2_{m} (Q_1 \wedge_R \hat{Q}_2).
$$
Hence $W$ is outer if and only if $Q_1 \wedge_R \hat{Q}_2=I_m$.
\end{proof}

Hence $W(z)\in \Hb_{(p+m),m}^2$ can be   outer even if  none of the two submatrices $W_1$ and $W_2$ is.
They just need to have no non-identity inner divisors in common (no common unstable zeros when $m=1$). On the other hand, when $W_1$ or $W_2$ have no unstable zeros, they are automatically outer.

Next is a proposition discussing the Wiener filter from $y_2(t-1)$ to $y_1(t)$ when $W_2$ is minimum-phase.

\begin{theorem}\label{lemcalF}
  Assume that $W_2$ is minimum-phase. Then  there is a representation \eqref{fbspecial} where  $F$  is stable and strictly causal, that is
$F(z)=z^{-1}\bar{F}(z)$ with $\bar{F}(z)$ causal and stable (analytic in $\{|z| \geq 1\}$) and $K(z)$ is a constant matrix $K_+$. In fact, this
$\bar{F}(z)$
coincides with  the transfer function $F_+(z)$  of the  one-step ahead Wiener predictor based on the strict past of $y_2$,  that is
\begin{equation}\nonumber
F_+(z) y_2(t-1) =\mathbb{E}\{y_1(t)\mid \Hb_{t-1}^-(y_2)\}
\end{equation}
and  the prediction error $\tilde y_1(t):= y_1(t)-F_+(z) y_2(t-1)$ can be written  $K_+ e(t)$ where $e(t)$ is the innovation of   the  joint process $y$.  The  representation
\begin{equation}\nonumber 
y_1(t)= F_+(z) y_2(t-1) +K_+ e(t)
\end{equation}
is  the unique feedback  representation \eqref{y1Fy2v} of $y_1(t)$ in which $e(t)$ is uncorrelated with the strict past of $y_2$.
\end{theorem}

\section{Proof of Lemma~\ref{thmUnique}}\label{apdxA.1}
\begin{proof}
A rational rank deficient spectral density corresponds to a stationary process (of dimension $m+p$ in our case). Since a stable $(m+p)\times m$ spectral factor always exists for such processes, the theorem can be proved by its equivalent statement: \\
{\it In all full-column-rank stable spectral factors of $\Phi$, there is one minimum-phase spectral factor unique up to right multiplication by an arbitrary $m\times m$ constant unitary matrix}.\\
Existence: Given any stable $p \times m$ spectral factor $W_o$ with rank $m$, we can perform an outer-inner factorization such that $W_o(z)=W(z)Q(z)$, where $Q(z)$ is an $m\times m$ inner function, to extract the minimum-phase factor of $W_o$.\\
Uniqueness: Suppose we have a minimum-phase $(m+p) \times m$ factor $W(z)$, 
and there is an square inner function $\hat{Q}(z)$, satisfying $\hat{W}(z)=W(z)\hat{Q}(z)$ is minimum-phase. Hence $\hat{Q}(z)$ must also be an outer function. Then by $\hat{Q}^{-1}(z)=\hat{Q}(z^{-1})^\top$, the function $\hat{Q}(z^{-1})$ is minimum-phase. Since the poles and zeros of $\hat{Q}(z^{-1})$ are the inverse of those of $\hat{Q}(z)$, $\hat{Q}(z)$ can only be a constant unitary matrix.
\end{proof}

\section{Uniqueness of the solution to some Sylvester matrix equation}\label{apdxB}
Though the two equations \eqref{Lyapct} and \eqref{Steindt} used in the spectral factorization in continuous time and discrete time are different, they are both Sylvester-type equations.

\begin{theorem}[\cite{Chu87}]\label{thmSylvUnique}
A Sylvester-type equation
$$
AXB^\top +CXD^\top =E
$$
has a unique solution if and only if (i) the matrix pencils $A-\lambda C$ and $D-\lambda B$ are regular; and \\
(ii) the spectrum of one (of the two pencils) is disjoint from the negative of the spectrum of the other. 
\end{theorem}

\section{More details on \eqref{W2Q1dt} being able to represent all the minimum-phase factors}\label{apdxC}

In this section, we shall prove that \eqref{W2Q1dt} can represent all the minimum-phase factors (in discrete time), without referring to Theorem~\ref{thmUnique}.

When equation \eqref{Steindt} has a unique invertible solution,
from \cite[Theorem 6.2]{Oara99}, $W_2(z)^\top$ and $Q_1(z)^\top$ from \eqref{W2dt} and \eqref{Q1dt} can represent all the solutions to the left-coprime factorization of $G_1(z)^\top H(z)^\top$ with an inner denominator.
Moreover, recall that from Theorem~\ref{thmHW1}, $H(z)$ is unique regardless of specific factor.
Hence we have to prove that, any minimum-phase factor of $\Phi_{11}$ can lead to all the tall full-rank minimum-phase factors given a fixed partition.

Suppose $G_1$ and $\hat{G}_1$ are two minimum-phase factors of $\Phi_{11}$, satisfying $\hat{G}_1=G_1\hat{P}$, where $\hat{P}$ is an $m\times m$ unitary matrix.
Suppose $G_1^\top H^\top=(Q_1^\top)^{-1}W_2^\top$, with $W_2,~Q_1$ in \eqref{W2dt}\eqref{Q1dt}, is the left-coprime factorization of $G_1^\top H^\top$ with an inner factor. Then we have a left-coprime factorization
\begin{equation}\nonumber
\begin{split}
   \hat{G}_1^\top H^\top &= \hat{P}^\top{G}_1^\top H^\top
     = \hat{P}^\top (Q_1^\top)^{-1}W_2^\top\\
     &=((\hat{P}^{-1}Q_1)^\top)^{-1}W_2^\top,
\end{split}
\end{equation}
where $\hat{Q}_1:=\hat{P}^{-1}Q_1$ is obviously an inner matrix. Then a minimum-phase full-rank factor can be represented as
\begin{equation}\nonumber
\begin{split}
      &W(z)=\begin{bmatrix}
           \hat{G}_1(z)\hat{Q}_1(z) \\
           W_2(z)
         \end{bmatrix}\\
         &=\begin{bmatrix}
           {G}_1(z)\hat{P}\hat{P}^{-1}{Q}_1(z) \\
           W_2(z)
         \end{bmatrix}
         =\begin{bmatrix}
           G_1(z)Q_1(z) \\
           W_2(z)
         \end{bmatrix},
\end{split}
\end{equation}
showing that any $G_1$ leads to the same solutions of minimum-phase tall full-rank factors. Hence Theorem~\ref{thmSPct} gives all the full-rank spectral factors of \eqref{Phi}.

\section{A numerical example of calculating a minimal realization \eqref{minRlz} from $T(s)$}\label{apdxExp}
Suppose we have a rational function in the continuous-time case,
\begin{align}\nonumber
    T(s)=\begin{bmatrix}
        \frac{(s+1)^2(s+2)^2}{(s+4)^2(s^2-2s+2)}, & \frac{(s+1)(s+3)}{(s-1)(s+4)}
    \end{bmatrix}^\top
\end{align}
with unstable distinct poles $1$, $1+i$, $1-i$, and a stable repeated pole $-4$ of degree $2$.

First we rewrite $T(s)$ in a summation form, and obtain 
\begin{align}\nonumber
    T(s) = & \frac{1}{s-1-i}\begin{bmatrix} 25(12-5i)/{338}\\0 \end{bmatrix} + \frac{1}{s-1+i} \begin{bmatrix} {25(12+5i)}/{338}\\0 \end{bmatrix}
    +\frac{1}{s-1}\begin{bmatrix} 0\\ {8}/{5}\end{bmatrix} \\ \nonumber
    & +\frac{1}{(s+4)^2}\begin{bmatrix} -{6(50s+161)}/{169} \\-{3(s+4)}/{5} \end{bmatrix}
    + \begin{bmatrix}1\\1 \end{bmatrix}.
\end{align} 

For the first two items in the summation, which contain a pair of complex poles symmetric on the both sides of the real axis, we use the controllable standard form to calculate a small minimal realization of them. We have 
\begin{align}\nonumber
    T_1(s):=&\frac{1}{s-1-i}\begin{bmatrix} {25(12-5i)}/{338}\\0 \end{bmatrix} + \frac{1}{s-1+i} \begin{bmatrix} {25(12+5i)}/{338}\\0 \end{bmatrix} \\ \nonumber
    =& \frac{1}{s^2-2s+2}\begin{bmatrix}
        {25(12s-7)}/{169} \\ 0
    \end{bmatrix} =\left[\begin{array}{cc|c}
    0 &  1 & 0\\ 
    -2 & 2  & 1\\ \hline
    {-175}/{169} & {300}/{169}  & 0 \\
    0 & 0 & 0
    \end{array}\right].
\end{align} 
For the simplicity of solving the following Lyapunov equation \eqref{Lyapct},we rewrite $T_1(s)$ as
\begin{align}\nonumber
    T_1(s)=\left[\begin{array}{cc|c}
    0 &  1 & 0\\ 
    -2 & 2  & {25}/{169}\\ \hline
    -7 & 12  & 0 \\
    0 & 0 & 0
    \end{array}\right].
\end{align} 

For the third item containing a distinct real pole in the above summation, by Gilbert realization, we have 
\begin{align}\nonumber
    T_2(s):=\frac{1}{s-1}\begin{bmatrix} 0\\ {8}/{5}\end{bmatrix} = \left[\begin{array}{c|cc}
    1 & {8}/{5} \\ \hline 
    0 & 0 \\
    1 & 0
    \end{array}\right].
\end{align}

For the forth item in the summation, denote its realization by
\begin{align}\nonumber
    T_3(s) = \left[\begin{array}{c|c}
      A_3 & B_3 \\ \hline
      C_3 & \begin{bmatrix}
          0\\0
      \end{bmatrix}
    \end{array}\right].
\end{align}
Since the degree of the repeated pole $-4$ is $2$, we let $A_3$ be a $2\times 2$ Jordan block,
\begin{align}\nonumber
    A_3=\begin{bmatrix}
        -4 & 1 \\
        0 & -4
    \end{bmatrix}.
\end{align}
Then it is easy to obtain matrices $B_3$, $C_3$ by solving equations. One of the solutions is
\begin{align}\nonumber
    T_3(s)=\left[\begin{array}{cc|c}
    -4 & 1  & {10}/{13}\\ 
    0 & -4 & -{3}/{5}\\ \hline 
    -{30}/{13} & 0 & 0 \\
    0 & 1 & 0
    \end{array}\right].
\end{align} 

Finally, we have 
\begin{align}\label{exp_ctrlz}
\begin{split}
    T(s)&= T_1(s)+T_2(s)+T_3(s) +\begin{bmatrix}
        1 & 1
    \end{bmatrix}^\top \\
    &=\left[\begin{array}{ccccc|c}
        0 &  1 & 0 & 0 & 0   & 0\\
        -2 & 2 & 0 & 0 & 0   & {25}/{169}\\
        0 & 0 & 1 & 0 & 0   & {8}/{5}\\
        0 & 0 & 0 & -4 & 1   &  {10}/{13}\\
        0 & 0& 0 & 0 & -4   & -{3}/{5}\\ \hline
        -7 & 12 & 0 & -{30}/{13} & 0 & 1 \\
        0 & 0 & 1 & 0 & 1 & 1
    \end{array}\right].
\end{split}
\end{align}

\bibliographystyle{siamplain}
\bibliography{CL23ver3}
\end{document}